\documentclass[11pt,leqno
]{amsart}
\usepackage{amsmath,epsfig,graphicx,color}
\usepackage{mathrsfs}
\definecolor{darkblue}{rgb}{.2, 0.2,.8}
\definecolor{carageen}{rgb}{0,0.5,0.3}
\definecolor{darkred}{rgb}{.8, .1,.1}

\newtheorem{lemma}{Lemma}[section]

\newtheorem{theorem}[lemma]{Theorem}

\newtheorem{proposition}[lemma]{Proposition}
\newtheorem{definition}[lemma]{Definition}
\newtheorem{corollary}[lemma]{Corollary}
\newtheorem{example}[lemma]{Example}
\newtheorem{exercise}[lemma]{Exercise}
\newtheorem{remark}[lemma]{Remark}
\newtheorem{fig}[lemma]{Figure}
\newtheorem{tab}[lemma]{Table}

\newcommand{\bfw}{{\bf w}}

\newcommand{\bth}{\begin{theorem}}
\newcommand{\ethe}{\end{theorem}}

\newcommand{\bre}{\begin{remark}\em }
\newcommand{\ere}{\end{remark}}

\newcommand{\ble}{\begin{lemma}}
\newcommand{\ele}{\end{lemma}}

\newcommand{\bde}{\begin{definition}}
\newcommand{\ede}{\end{definition}}
\newcommand{\bco}{\begin{corollary}}
\newcommand{\eco}{\end{corollary}}

\newcommand{\bpr}{\begin{proposition}}
\newcommand{\epr}{\end{proposition}}

\newcommand{\bexer}{\begin{exercise}}
\newcommand{\eexer}{\end{exercise}}

\newcommand{\bexam}{\begin{example}}
\newcommand{\eexam}{\end{example}}

\newcommand{\bfi}{\begin{fig}}
\newcommand{\efi}{\end{fig}}

\newcommand{\btab}{\begin{tab}}
\newcommand{\etab}{\end{tab}}

\newcommand{\beao}{\begin{eqnarray*}}
\newcommand{\eeao}{\end{eqnarray*}\noindent}

\newcommand{\beam}{\begin{eqnarray}}
\newcommand{\eeam}{\end{eqnarray}\noindent}

\newcommand{\beqq}{\begin{equation}}
\newcommand{\eeqq}{\end{equation}\noindent}

\newcommand{\bce}{\begin{center}}
\newcommand{\ece}{\end{center}}

\newcommand{\barr}{\begin{array}}
\newcommand{\earr}{\end{array}}

\newcommand{\vague}{\stackrel{\lower0.2ex\hbox{$\scriptscriptstyle
                    \it{v} $}}{\rightarrow}}
\newcommand{\weak}{\stackrel{\lower0.2ex\hbox{$\scriptscriptstyle
                    \it{w} $}}{\rightarrow}}
\newcommand{\what}{\stackrel{\lower0.2ex\hbox{$\scriptscriptstyle
                    \it{\hat{w}} $}}{\rightarrow}}

\newcommand{\bdis}{\begin{displaymath}}
\newcommand{\edis}{\end{displaymath}\noindent}

\newcommand{\N}{\mathbb{N}}
\newcommand{\R}{\mathbb{R}}

\newcommand{\ov}{\overline}
\newcommand{\wt}{\widetilde}

\newcommand{\bbC}{{\mathbb C}}

\newcommand{\E }{{\mathbb E}}
\newcommand{\Z }{{\mathbb Z}}
\renewcommand{\P }{{\mathbb P}}

\allowdisplaybreaks

\begin{document}
\today
\bibliographystyle{plain}
\title[Arrival times of Cox process with independent increments]{
Arrival times of Cox process with independent increments \\
with application to prediction problems}
\thanks{Muneya Matsui's research is partly supported by JSPS
Grant-in-Aid for Young Scientists B (16K16023)}
\author[M. Matsui]{Muneya Matsui}
\address{Department of Business Administration, Nanzan University, 18
Yamazato-cho, Showa-ku, Nagoya 466-8673, Japan.}
\email{mmuneya@gmail.com}

\begin{abstract}
 Properties of arrival times are studied for a Cox process with 
 independent (and stationary) increments. Under a reasonable
 setting the directing random measure 
 is shown to take over independent (and stationary) increments of the process, 
 from which the sets of arrival times and their numbers in
 disjoint intervals are proved to be independent (and
 stationary). Moreover, we derive the exact joint distribution of these
 quantities with Gamma 
 random measure, whereas for a general 
 random measure the method of calculation is presented. Based on the
 derived properties we consider prediction problems for the shot noise process
 with Cox process arrival times which trigger additive processes off. 
 We obtain a numerically tractable expression for the predictor which works reasonably in
 view of numerical experiments.
\end{abstract}
\maketitle

\section{Preliminaries} 
In this paper we consider a Cox process $N:=(N(t))_{t\ge 0}$ directed by
random measure $\eta:=(\eta(t))_{t\ge 0}$ such that $\eta$ is a non-decreasing
c\`adl\`ag process with $\eta(0)=0\ a.s.$ (cf. \cite[p.44]{grandell:1997}). In particular we investigate
distributional properties of arrival times $0<T_1 \le T_2 \le \cdots$ of
$N(t)$. 
As an application we consider prediction problems for a shot noise-type process
 \begin{align}\label{eq:cpp}
 M(u)
 =\sum_{j=1}^{N(u)} L_{j}(u-T_{j})\,, \quad u\ge 0\,,
 \end{align}
 where $(L_j),\,j=1,2,\ldots$ are 
 independent identically distributed (iid) processes with $L_j(u)=0\ a.s.$ for
 $u\le 0$ such that $(L_j)$ are independent of $(N,\,\eta)$. 
 
 Early on, properties related with arrival times of Poisson processes have been intensively
 studied and now we can find comprehensive theories in textbooks
 e.g. \cite{kallenberg:1983,grandell:1997,embrechts:kluppelberg:mikosch:1997,mikosch:2009}. Among them the order statistics
 property is one of the most important properties, which characterizes
 the Poisson process and which is satisfied
 only with the mixed Poisson process. See e.g. \cite[Theorem 9.1, Corollary
 9.2]{kallenberg:1983} or \cite{matthes:kerstan:mecke:1978}, both of
 which treat general (non-diffuse) mean measures.  
 However, although the Cox process is a
 direct generalization of the mixed Poisson process, few attempts have
 been made to arrival times of the Cox process.   
 One reason is that we have to treat random intensities which yield not
 a little complexity in the analysis. 

 In this paper we derive properties for arrival times of Cox process $N$ assuming
 independent (and stationary) increments for $N$. More precisely,
 let $0=t_0<t_1<\cdots<t_k<\infty$ and we show that sets of
 points $(T_\ell) \in (t_{j-1},t_j]$ of $N$ and their numbers $N(t_{j-1},t_j]:=N(t_j)-N(t_{j-1})$ in
 disjoint intervals $(t_{j-1},t_j],\,j=1,\ldots,k$ are mutually independent. The result
 is based on the secondary result that the directing random measure
 $\eta$ should succeed to independent (and 
 stationary) increments from $N$ under
 reasonable assumptions. 
 We also investigate joint distributions of points $(T_\ell) \in (t_{j-1},t_j]$ and
 the number $N(t_{j-1},t_j]$ 
 based on derived independent (and stationary) increments for $\eta$. 
 An explicit expression for these
 quantities is obtained when $\eta$ is given by a gamma process, while
 for general random measure a method of the calculation is obtained. 

 As an application we consider a prediction problem of the model \eqref{eq:cpp} given the
 past information assuming that $\eta$ is an additive or a L\'evy
 process. Make a good use of the obtained results about arrival times, 
 we present numerically reasonable expressions for predictors. In view
 of the numerical experiments examined, the method seems to work reasonably.  

 From old times up to the present, shot noise processes have been
 providing attractive stochastic models for describing both natural and social
 phenomena (see the survey \cite{bondesson:2006}). In fact, recent applications are found in e.g. managing the
 workload of large computer networks
 \cite{fay:gonzalez:mikosch:samorodnitsky:2006,mikosch:samorodnitsky:2007},
 financial modeling
 \cite{kluppelberg:kuhn:2004,Kluppelberg:matsui:2015,schimidt:2017} and
 modeling delay in claim settlement of non-life insurance
 \cite{mikosch:2009}. Although there have been some researches on
 prediction of the processes e.g. \cite{LBPJ:1999}, the active studies of the topic started fairly
 recently (see 
 Section \ref{sec:prediction}).  
 
 Finally we make two remarks. 
 First our motivation is to construct
 structure-based predictors, which is more than just applying ready-made linear
 predictions (see introduction of \cite{matsui:mikosch:2010}). For this
 aim properties of arrival times are crucial, and 
 independent increments assumption is our starting point. Although
 derived properties here are more restrictive than those
 of mixed Poisson (see \cite{matsui:rolski:2016}), they 
 still keep the core tool (order statistics-type property) to compute 
 predictors. It would be our challenging future work to investigate arrival times in
 a more general setting, where we could not resort to this nice property anymore.  
 In our recognition the independent increments would be the best
 possible to exploit the property. 

 Second Cox processes with $\eta$ to be an additive process or
 subordinator are regarded as random time changes of L\'evy processes and these Cox
 processes are known to have independent (and stationary) increments (see
 \cite{barndorffnielsen:shiryaev:2015,bochner:2005,sato:1999}). Therefore we obtain a kind of reverse result,
 i.e. assuming independent (and stationary) increments of the time
 changed non-negative additive or L\'evy processes, we show that the underlying
 process should be a subordinator.   

 Throughout, we assume that processes $\eta$ and $L$ are 
 additive processes. 
 An additive process $V:=(V(u))_{u\ge
 0}$ has independent
 increments with c\`adl\`ag path starting at $V(0)=0$ a.s. 
 and is stochastically continuous. 
 The distribution of the process $V$ is completely determined by the system of generating triplet
 $\{(A_u,\rho_u,\gamma_u):u\ge 0\}$ via characteristic function. See
 \cite[Remark 9.9]{sato:1999} for more details. We will restrict ourselves to processes
 having only a pure jump part (see \cite[Theorem 19.3]{sato:1999}), namely    
 we consider $V$ such that
 \begin{align}
\label{chf:additive:process}
 E[ e^{ix V(u)}]=\exp\Big\{
 \int_{(0,u]\times \mathbb{R}} (e^{ixv}-1)\, \rho (d (w,v))
 \Big\},
 \end{align}
 where a measure $\rho$ on $(0,\infty)\times \mathbb{R}$ satisfies
 $\rho ((0,\cdot]\times \{0\})=0$, $\rho (\{u\}\times \mathbb{R})=0$ and
\begin{align}
\label{eq:condi:levym}
\int_{(0,u]\times \mathbb{R}} (|w|\wedge 1)\, \rho (d (w,v))<\infty,\qquad
 \mathrm{for}\ u \ge0.  
 \end{align}
 In this case the generating triplet is $(0,\rho_u,0)$ with $\rho_u(B):=\rho ((0,u]\times B)$ for
 any Borel set $B\in\mathcal B(\mathbb{R})$. 
 
 Among additive processes we frequently focus on 
 L\'evy processes 
 which additionally have stationary increments property
 , so that the
 system of generating triplet of $V$ is given by $\{u (A,\rho,\gamma):u\ge
 0\}$ (see \cite[Corollary 8.3]{sato:1999}). Particularly we are interested
 in the class of subordinators (cf. \cite[Theorem 30.1]{sato:1999}) of
 which Laplace transform (LP for short) is 
 \begin{align}
 \label{lh:represent}
  \E[e^{-xV(1)}] = \exp \big\{
 -\gamma x + \int_{(0,\infty)}(e^{-xv}-1)\nu(dv)\big\},\quad 
 x\ge 0
 \end{align} with $\gamma\ge 0$ and 
 \[
  \int_{(0,\infty)} (1\wedge s)\nu(ds) <\infty. 
 \]
 Note that subordinators could be random measures on $[0,\infty)$, since 
 we can construct the Lebesgue-Stieltjes measure pathwisely from 
 a.s. non-decreasing c\`adl\`ag process. 

 Throughout we use the following notations:
 $\Z_+:=\{0,1,2,\ldots\},\,\R_+:=[0,\infty)$. Moreover let  
 $\mathbb{C}$ be the field of complex numbers and 
 $\mathbb{C}^n:=\{(z_1,\ldots,z_n):z_v \in \mathbb{C}\, \mathrm{for}\, 1\le
 v\le n\}$. 
 As usual write $X\sim\cdot $ if r.v. $X$
follows the distribution after the tilde.

\section{Main result}

In this section we characterize properties of the directing random measure
$\eta$ from those of the Cox process $N$, namely, under some reasonable setting
we specify the class of directing 
measures $\eta$, which is shown to be that of non-negative additive or L\'evy
processes. In addition we show an order statistics-type property for the Cox
processes driven by additive or L\'evy
processes. Based on those main results the method of calculating the joint distribution of arrival
times and their number is presented. Moreover, 
the exact joint distribution is obtained when $\eta$ is a gamma process
 random measure.

Notice that Cox process of this type can be regarded as the time-changed
L\'evy process by a subordinator which we call subordination
(see \cite[Definition 30.2]{sato:1999}). In our case
the original L\'evy process is given by a homogeneous Poisson.
In the literature of subordination, part of our results are known. We
will state the detailed relation with past results after
Remark \ref{rem:orderstatistics}. 

We are starting to see the order
statistics-type property (Theorem \ref{thm:main1} and Corollary
\ref{corollary:main1}). 
 

\begin{theorem}
\label{thm:main1}
 Let $N$ be a Cox process directed by a non-decreasing c\`adl\`ag
 process $\eta$ with $\eta(0)=0$, which is stochastically continuous. 
 Denote arrival times of $N$ by $0\le T_1 \le
 T_2 \le \cdots$. Let $0=t_0<t_1<t_2<\cdots< t_k \le t,\,k\in\N,\,t>0$. We
 further denote finer points in the intervals $(t_{j-1},t_j],\,j=1,2,\ldots,k$
 by 
 $t_{j-1}<t_{j1}\le \cdots \le t_{j\ell_j} \le t_j,\,\ell_j\in\Z_+$
 and write 
 $s_k=\sum_{j=1}^k \ell_j$. Then \\
 $($i$)$ for any $k\in \N,\,\ell_j\in \Z_+,\,j\le k$,
 \begin{align}
 \label{eq:indepincre}
  \P\big( \cap_{j=1}^k N(t_{j-1},t_j]=\ell_j \big) = \prod_{j=1}^k
  \P\big(
 N(t_{j-1},t_j] =\ell_j
 \big)
 \end{align}
 holds if and only if $\eta$ is an additive process, and in this case,  
\begin{align}
& \P\big(
\cap^k_{j=1}\{T_{s_{j-1}+1} \le t_{j1},\ldots,T_{s_j} \le
t_{j\ell_j} 
,\,N(t_{j-1},t_j]=\ell_j \}
\big) \nonumber \\
 & = \prod_{j=1}^k \P \big(
T_{s_{j-1}+1}\le t_{j1},
\ldots,T_{s_j}\le
 t_{j\ell_j},\,N(t_{j-1},t_j] =\ell_j 
\big). \label{eq:condiid}
\end{align}
$($ii$)$ the left-hand side of \eqref{eq:indepincre} is equal to 
\begin{align} \label{eq:stationary}
 \prod_{j=1}^k \P\big(
 N(0,t_j-t_{j-1}]=\ell_j
\big)
\end{align}
if and only if $\eta$ is a subordinator, and in this case the left-hand
 side of \eqref{eq:condiid} is equal to 
\begin{align}
\label{eq:condiidlv}
 \prod_{j=1}^k \P ( T_1\le t_{j1}-t_{j-1},
 \ldots, T_{\ell_j}\le
 t_{j\ell_j}-t_{j-1},N(0,t_j-t_{j-1}]=\ell_j).  
\end{align}
\end{theorem}

\begin{proof}
 (i) `if part' Given $\eta$ the joint conditional distribution of
 $(N(t_{j-1},t_j])_{j\le k}$ has 
\begin{align} \label{pf:11}
 \P\big(
 \cap_{j=1}^k N(t_{j-1},t_j]=\ell_j \mid \eta
 \big) = \prod_{j=1}^k \frac{\eta(t_{j-1},t_j]^{\ell_j}}{\ell_j!}
 e^{-\eta(t_{j-1},t_j]}. 
\end{align}
 Since $\eta$ has independent increments by taking expectation with
 respect to $\eta$ we obtain \eqref{eq:indepincre}. \\
`only if part' For the proof, we consider the LP of 
 $(\eta(t_{j-1},t_j])_{j\le k}$ and show 
\begin{align}
\label{pfeq:ind0}
 \E \prod_{j=1}^k e^{-\alpha_j \eta(t_{j-1},t_j]}  = \prod_{j=1}^k \E
 e^{-\alpha_j \eta(t_{j-1},t_j]},\quad \alpha_j\ge0,\,j=1,\ldots,k. 
\end{align}
Since \eqref{eq:indepincre} implies that for any $\ell_j \in
 \Z_+,\,j=1,\ldots,k$ 
 \begin{align}
 \label{pfeq:ind1}
  \E \prod_{j=1}^k \frac{\eta(t_{j-1},t_j]^{\ell_j}}{\ell_j!}
 e^{-\eta(t_{j-1},t_j]}= \prod_{j=1}^k \E
 \frac{\eta(t_{j-1},t_j]^{\ell_j}}{\ell_j!} e^{-\eta(t_{j-1},t_j]},   
 \end{align}
 using the expansion 
\begin{align*}
 \prod_{j=1}^k e^{-\alpha_j\eta(t_{j-1},t_j]} = \prod_{j=1}^k \sum_{\ell_j =0}^\infty
 \frac{(1-\alpha_j)^{\ell_j} }{\ell_j !} \eta(t_{j-1},t_j]^{\ell_j}
 e^{-\eta(t_{j-1},t_j]}\quad a.s., 
\end{align*} 
it seems that we may apply Fubuni's theorem and could easily obtain
 \eqref{pfeq:ind0}. Here the expansion is possible by c\`adl\`ag assumption on $\eta$. 
 However, Fubini's theorem could be applicable only when
 $|1-\alpha_j|<1,\,j=1,\ldots,k$, since we need to exchange expectation and the
 infinite sums. Therefore, \eqref{pfeq:ind0} holds only on some subset of
 $\R_+^k$. 

 To overcome this difficulty, we extend the domain of the LP,
 $(\alpha_1,\ldots,\alpha_k)\in \R_+^k$ to $k$-dimensional complex plane
 with positive axes denoted by $\mathbb{C}^k_{++}:=\{\bfw \in \mathbb{C}^k \mid \mathrm{Re}\, w_j >
 0,j=1,\ldots,k\}$ and apply the identity theorem in several complex
 variables. 

 In what follows, we show  
 \begin{align}
\label{pfeq:ind2}
  L_1(\bfw):= \E \prod_{j=1}^k e^{-w_j \eta(t_{j-1},t_j]}  = \prod_{j=1}^k \E
 e^{-w_j \eta(t_{j-1},t_j]} =L_2(\bfw) 
 \end{align}  
 on the whole $\bfw\in \mathbb{C}_{++}^k$. Then
 continuity of $L_i$ on exes and at the origin, \eqref{pfeq:ind2}
 holds on all $\bfw \in \R_+^k$. 
Due to the identity theorem, e.g. \cite[Theorem
 4.1, Ch.1]{grauert:fritzsche:1976}, it suffices to show that 
\begin{itemize}
 \item (C1) : $g_1,\,g_2$ are holomorphic in $\mathbb{C}^k_{++}$.
 \item (C2) : There is a nonempty region (an open set) $B\subset
 \mathbb{C}^k_{++}$ such that \eqref{pfeq:ind2} holds on $B$.  
\end{itemize}
Condition (C2) is easy if we take $B=\{\bfw\in \mathbb{C}_{++}^k\mid
 |w_j-1|<1,j=1,\ldots,k \}$. Indeed, 
 twice applications of Fubini's theorem yield 
 \begin{align*}
  L_1(\bfw) 
  &= \E \prod_{j=1}^k
  \sum_{\ell_j =0}^\infty
  \frac{(1-w_j)^{\ell_j}\,\eta(t_{j-1},t_j]^{\ell_j}}{\ell_j !}e^{-\eta(t_{j-1},t_j]} \\
 &= \sum_{\ell_1,\ldots,\ell_k\ge 0}\E
  \prod_{j=1}^k\frac{(1-w_j)^{\ell_j}\,\eta(t_{j-1},t_j]^{\ell_j}}{\ell_j
  !}e^{-\eta(t_{j-1},t_j]}  \\
 &= \sum_{\ell_1,\ldots,\ell_k\ge 0}
  \prod_{j=1}^k \E
  \frac{(1-w_j)^{\ell_j}\,\eta(t_{j-1},t_j]^{\ell_j}}{\ell_j
  !}e^{-\eta(t_{j-1},t_j]}  \\
 & = \prod_{j=1}^k \Big( \E \sum_{\ell_j=0}^\infty
  \frac{(1-w_j)^{\ell_j}\,\eta(t_{j-1},t_j]^{\ell_j}}{\ell_j !}e^{-\eta(t_{j-1},t_j]}\Big)=L_2(\bfw) 
 \end{align*}
on the region $B\in \mathbb{C}_{++}^k$.  

For Condition (C2), it suffices to show that both $L_1$ and $L_2$ are
 complex differentiable on $\mathbb{C}^k_{++}$ which is equivalent to ``holomorphic'', e.g. 
\cite[Theorem 3.8, Ch.1]{grauert:fritzsche:1976}. Thus, we check that
 $L_1$ and $L_2$ satisfy the Cauchy-Riemann differential equation in
 each component, cf. \cite[Theorem 6.2, Ch.1]{grauert:fritzsche:1976},
 which are $\partial L_l/\partial \ov w_j=0,\,l=1,2,\,j=1,\ldots,k$ on
 $w_j$. Since we have the existence of 
\[
 \E \sum_{j=1}^k \eta(t_{j-1},t_j] \prod_{j=1}^k
 e^{-w_j\eta(t_{j-1},t_j]} 
\]
for any $\bfw\in\mathbb{C}_{++}^k$, we could exchange the order of
 derivative and expectation and obtain 
\begin{align*}
 \frac{\partial L_1}{\partial \overline w_j} = \E \frac{ \partial
  \prod_{l=1}^k e^{-w_l \eta(t_{l-1},t_l]}}{\partial \overline w_j}=0.
\end{align*}
Here we use the relation $\mathrm{Re}\,L_1=(L_1+\ov L_1)/2$ and
 $\mathrm{Im}\,L_1= -i(L_1-\ov L_1)/2$. The proof of $\partial
 L_2/\partial \ov w_j=0$ on $\mathbb{C}_{++}^k$ is similar. 

 Next we show \eqref{eq:condiid}. 
We omit the case $\ell_k=0$ for some $k$,
 since the proof is easier, and we always assume $\ell_j\ge 1$ for all $j$. 
By the order statistics property of Poisson
 (e.g. \cite[Theorem 6.6]{grandell:1997} or \cite[Corollary 9.2]{kallenberg:1983}) given
 $\{N(t)=n,\,\eta(0,u],\,0<u\le t\}$, the conditional distribution
 of $(T_1,\ldots,T_n)$ is by those of the order statistics of iid
 samples with common distribution $\eta(d x)/\eta(t),\,0<x\le t\,a.s.$
 Moreover, given $\{N(s)=m,\,N(s,s+t]=n-m,\,\eta(u),\,0<u\le s+t\}$, $(T_1,\ldots,T_m)$ on $(0,s]$ and $(T_{m+1},\ldots,T_n)$ on
 $(s,s+t],\,s,t>0$ are independent and respectively have distributions
 of the order statistics with common distributions $\eta(dx)/\eta(s)$ and
 $\eta(dy)/\eta(s,s+t]$, $0<x \le s <y \le s+t$ a.s. Hence due to the
 distribution function (d.f. for short) of order
 statistics under discontinuous d.f. (see Proof of Theorem
 1.5.6. \cite{reiss:1989}), we have 
 \begin{align}
\label{eq:condidist}
  & \P\big(
 T_1 \le t_{11},\ldots,T_{s_k} \le
  t_{k\ell_k},\,\cap_{j=1}^k \{ N(t_{j-1},t_j]=\ell_j \} \mid
  \eta(u),\,0<u\le t
 \big) \\
 & = \prod_{j=1}^k \E \big[
 {\bf
  1}_{\times_{i=1}^{\ell_j}(t_{j-1},t_{ji}]}(T_{s_{j-1}+1}',\ldots,T_{s_j}')\,\ell_j
  !\,
  {\bf
  1}_{\{H_j(U_{s_{j-1}+1},T_{s_{j-1}+1}')<\cdots<
  H_j(U_{s_j},T_{s_j}')\}}\mid \eta \big] \nonumber \\
 &\qquad \times \frac{\eta(t_{j-1},t_j]^{\ell_j}}{\ell_j!}
  e^{-\eta(t_{j-1},t_j]} \nonumber \\
 &=\prod_{j=1}^k \P\big(
 T_{s_{j-1}+1}'\le t_{j1},
 \ldots,T_{s_{j}}'\le t_{j\ell_j},\,N(t_{j-1},t_j]=\ell_j \mid
  \eta(u),\,t_{j-1}<u \le t_j 
 \big),\nonumber
 \end{align}
where 
\[
 H_j(x,y)= \frac{\eta(t_{j-1},y)}{\eta(t_{j-1},t_j]}+x
 \frac{\eta(\{y\})}{\eta(t_{j-1},t_j]},\quad t_{j-1}<y \le t_j 
\]
and conditionally iid r.v.'s $(T_{s_{j-1}+1}',\ldots,T_{s_j}')$ possess the common
 d.f. $\eta(dx)/\eta(t_{j-1},t_j],\,t_{j-1}<x\le t_j$, and
 $(U_j)_{j=1,2,\ldots}$ are iid uniform $U(0,1)$ r.v.'s
 independent of everything. Here $\eta(t_{j-1},y)$ means
 $\eta(y-)-\eta(t_{j-1})$ with $\eta(y-):=\lim_{x \uparrow y}\eta(x)$. 
 We notice
 that in the last line, each $j$ term of the product is included
 in the $\sigma$-field by $\{\eta(u),\,t_{j-1}<u\le t_j\}$,
 and they are 
 independent in $j$. Hence taking expectation with $\eta$ in both sides
 of \eqref{eq:condidist}, we
 obtain \eqref{eq:condiid}. 


\noindent
(ii) `if part' 
 Given $\eta$, the
 joint distribution of increments $(N(t_{j-1},t_j])_{j\le k}$, 
 i.e. \eqref{pf:11} is distributionally equal to 
\begin{align*}
 \prod_{j=1}^k \frac{\eta_j(0,t_j-t_{j-1}]^{\ell_j}}{\ell_j!}
 e^{-\eta_j(0,t_j-t_{j-1}]} = \prod_{j=1}^k \P \big(
 N_j(0,t_j-t_{j-1}]=\ell_j \mid \eta_j
\big),
\end{align*}
where $(\eta_j)$ are iid copies of $\eta$ and $(N_j)$ are iid Cox processes
 with random measures $(\eta_j)$ respectively, so that finite
 dimensional distributions of $N_j$ coincide with those of $N$. The
 result is implied by taking expectation with respect to $\eta$ and $(\eta_j)$. \\
`only if part' Assume \eqref{eq:stationary} and proceed as in the case
 $(i)$. Then by the stationary increments property of $N$, the
 relation \eqref{pfeq:ind1} is replaced with 
 \begin{align*}
  \E \prod_{j=1}^k \frac{\eta(t_{j-1},t_j]^{\ell_j}}{\ell_j!}
  e^{-\eta(t_{j-1},t_j]} = \prod_{j=1}^k
  \E \frac{\eta(0,t_j-t_{j-1}]^{\ell_j}}{\ell_j!} e^{-\eta(0,t_j-t_{j-1}]},
 \end{align*}
 so that we can obtain
 \[
  \E \prod_{j=1}^k e^{-\alpha_j\eta(t_{j-1},t_j]} = \prod_{j=1}^k \E
 e^{-\alpha_j\eta(0,t_j-t_{j-1}]} 
 \]
 for $k \in \N$. Comparing this with the last term in
 \eqref{pfeq:ind0}, we prove the stationary increments property of $\eta$. 

 Next we show \eqref{eq:condiidlv}. Since $\eta$ is a subordinator,
\[
 \big(
 H_j(x,y),\,\eta(t_{j-1},t_j]
\big)_{j\le k} \stackrel{d}{=} \big(\wt
 H_j(x,y),\,\eta_j(0,t_j-t_{j-1}]\big)_{j\le k},\quad t_{j-1}<y \le t_j,
\] 
 where $(\eta_j)$ is a sequence of iid copies of $\eta$ and 
 \[
  \wt H_j(x,y) = \frac{\eta_j(0,y-t_{j-1}]}{\eta_j(0,t_j-t_{j-1}]}+ x
 \frac{\eta_j(\{y-t_{j-1}\})}{\eta_j(0,t_j-t_{j-1}]}. 
 \]
 Hence \eqref{eq:condidist} is distributionally equal to 
 \begin{align*}
  & \prod_{j=1}^k \E \big[
 {\bf 1}_{\times_{i=1}^{\ell_j}(0,t_{ji}-t_{j-1}]}(\wt
  T_{s_{j-1}+1},\ldots,\wt T_{s_j})\,\ell_j!\, {\bf
  1}_{\{\wt H_j(U_{s_{j-1}+1},\wt T_{s_{j-1}+1})< \cdots< \wt H_j(U_{s_j},\wt
  T_{s_j})\} } \mid \eta_j
 \big] \\
 &\quad \times \frac{\eta_j(0,t_j-t_{j-1}]^{\ell_j}}{\ell_j!}
  e^{-\eta_j(0,t_j-t_{j-1}]},  
 \end{align*}
where given $(\eta_j)$, r.v.'s $(\wt T_{s_{j-1}+1},\ldots,\wt T_{s_j})$
 are conditionally iid and possess the common
 d.f. $\eta_j(dx)/\eta_j(0,t_j-t_{j-1}]$, namely they constitute points
 of $N_j(0,t_j-t_{j-1}]$.  Hence we may write \eqref{eq:condidist} as  
 \begin{align*}
 & \P\big(
 T_1 \le t_{11},\ldots,T_{s_k} \le
  t_{k\ell_k},\,\cap_{j=1}^k \{ N(t_{j-1},t_j]=\ell_j \}\mid
  \eta 
 \big) \\ 
 & = \prod_{j=1}^k 
\P (\wt T_{s_{j-1}+1}\le t_{j1}-t_{j-1},
 \ldots, \wt T_{s_j}\le t_{j\ell_j}-t_{j-1},\,
N_j (0,t_j-t_{j-1}]= \ell_j \mid \eta_j ). 
 \end{align*}
 Now taking expectation of both sides, we obtain \eqref{eq:condiidlv}. 
\end{proof}

The following results are immediate consequence from Theorem
\ref{thm:main1}.

\begin{corollary}
\label{corollary:main1}
 Suppose the same notations and conditions of Theorem \ref{thm:main1}
 until the line before the item $(i)$. Then\\
 $(i)$ Assume \eqref{eq:indepincre} or equivalently that $\eta$ is an
 additive process, then conditional joint distribution of points in
 disjoint intervals given the numbers are mutually independent, i.e. 
\begin{align}
 \label{eq:condi1}
 & \P\big(
 \cap_{j=1}^k \{T_{s_{j-1}+1} \le t_{j1},\ldots,T_{s_j} \le
 t_{j\ell_j}\} \mid \cap_{j=1}^k \{N(t_{j-1},t_j] =\ell_j \}
 \big)\\
 & = \prod_{j=1}^k \P\big(
 T_{s_{j-1}+1}\le t_{j1},\ldots,T_{s_j}\le t_{j\ell_j}  \mid
 N(t_{j-1},t_j]=\ell_j  \big). \nonumber
\end{align}
$(ii)$ Assume \eqref{eq:condiidlv} or equivalently that $\eta$ is a subordinator, then
 conditional on numbers of points in disjoint intervals, 
 points in each intervals further satisfy stationarity in a sense that 
 \begin{align}
\label{eq:condi2}
 & \P\big(
 \cap_{j=1}^k \{T_{s_{j-1}+1} \le t_{j1},\ldots,T_{s_j} \le
 t_{j\ell_j}\} \mid \cap_{j=1}^k N(t_{j-1},t_j] =\ell_j 
 \big)\\
 & = \prod_{j=1}^k \P\big(
 T_{1} \le t_{j1}-t_{j-1},\ldots,T_{\ell_j}\le t_{j\ell_j}-t_{j-1} \mid
 N(0,t_j-t_{j-1}]=\ell_j  \big). \nonumber
\end{align}
\end{corollary}

\begin{remark}
 \label{rem:orderstatistics}
 $(i)$ One may suspect that conclusions in Theorem \ref{thm:main1} still
 hold even when the
 condition \eqref{eq:indepincre} $($resp. \eqref{eq:stationary}$)$ is
 replaced by \eqref{eq:condi1} $($resp. \eqref{eq:condi2}$)$. However, there
 exists a counterexample by a mixed Poisson process whose random
 measure satisfies \eqref{eq:condi1} $($resp. \eqref{eq:condi2}$)$ and conditions of Theorem \ref{thm:main1} other than
 \eqref{eq:indepincre} $($resp. \eqref{eq:stationary}$)$
 , in the meanwhile, the
 process does not have independent increments $($cf. \cite[Appendix A]{matsui:rolski:2016}$)$. In other words, we can
 construct Cox processes without independent (and stationary) increments
 which satisfy \eqref{eq:condi1} $($and \eqref{eq:condi2}$)$. \\
 $(ii)$ In Theorem \ref{thm:main1} \eqref{eq:indepincre} $($resp. \eqref{eq:stationary}$)$
  is implied by \eqref{eq:condiid} $($resp. \eqref{eq:condiidlv}$)$
  which we see by putting $t_{jk}=t_j$ for all $k\in\ell_j$. Therefore,
 the three conditions,  
 \eqref{eq:indepincre}, \eqref{eq:condiid} and the additivity of $\eta$, are
 equivalent. Similarly the three conditions, \eqref{eq:stationary}, \eqref{eq:condiidlv}
 and the subordinator assumption on $\eta$, are equivalent. 
\end{remark}

Before going to the distribution of
$((T_j)_{j:T_j\le t},N(t))$, we give a more general result of 
Theorem \ref{thm:main1}, the proof of which gives an alternative proof for
the theorem, namely we derive the corresponding result for the time
changed L\'evy processes: Cox processes are obtained by selecting 
the standard Poisson for the original L\'evy process. 
Note that the time changed L\'evy processes by the
subordinators are shown to be L\'evy processes again (see e.g. \cite[Theorem
30.1]{sato:1999} and \cite[Section 4]{bochner:2005}). We show a kind of 
reverse relation which seems to be new.  

\begin{theorem}
\label{thm:general}
 Let $\eta$ be non-decreasing c\`adl\`ag process with $\eta(0)=0$,
 which is stochastically continuous, and let $L$ be a subordinator 
 whose 
 LP is given in \eqref{lh:represent}. 
 Then the time changed L\'evy process $N(t):= L(\eta(t))$ has \\
$(i)$ independent increments if and only if $\eta$ has independent
 increments. \\
$(ii)$ stationary independent increments if and only if $\eta$ is a
 subordinator. 
\end{theorem}

\begin{proof}
 Since the proof for $L$ with drift terms is
 similar and easier, throughout we assume that $L$ has no drift. 
 Let $0=t_0<t_1<t_2<\cdots<t_n \le t$ for $t>0$. \\
$(i)$ The independent increments property of $N$ follows easily from
 that of $\eta$ and therefore we show the reverse. Since $N$ has
 independent increments, we have for $u_k\ge0,\,k=1,\ldots,n$, 
\begin{align}
\label{eq:LPindep}
 \E \prod_{k=1}^n e^{\eta(t_{k-1},t_k]\int_{\R_+} (e^{- u_k x}-1)\nu(dx)} &= 
 \E \prod_{k=1}^n e^{-u_k N(t_{k-1},t_k]}\\
 &= \prod_{k=1}^n \E e^{-u_k N(t_{k-1},t_k]}\nonumber \\
&=\prod_{k=1}^n \E e^{\eta(t_{k-1},t_k]\int_{\R_+} (e^{- u_k x}-1)\nu(dx)}, \nonumber 
\end{align}
where we consider the conditional distribution of $N$ given $\eta$.
Formally putting $-w_k=f(u_k)=\int_{\R_+}(e^{-u_kx}-1)\nu(dx)$ in 
 \eqref{eq:LPindep}, we have 
\begin{align}
\label{eq:lpindep}
 g_1(\bfw) := \E \prod_{k=1}^n e^{-w_k \eta(t_{k-1},t_k]} = \prod_{k=1}^n
 \E e^{-w_k \eta(t_{k-1},t_k]} =: g_2 (\bfw),\quad \bfw=(w_1,w_2,\ldots,w_n).
\end{align}
Since $g_1$ is the joint LP of
 $(\eta(t_{k-1},t_k])_{k\le n}$ and $g_2$ is the product of
 those for $\eta(t_{k-1},t_k],\,k=1,\ldots,n$, if we show
 \eqref{eq:lpindep} for all $\bfw\in \R^n_+$ the desired independence
 follows. We rigorously show this by applying the identity theorem in
 several complex variables. 

 We prepare, as the domain of $\bfw$, $n$-dimensional complex plane with positive
 axis denoted by $\mathbb{C}^n_{++}:=\{\bfw \in \mathbb{C}^n \mid \mathrm{Re}\, w_k >
 0,k=1,\ldots,n\}$ and show that 
\begin{itemize}
 \item (C1) : $g_1,\,g_2$ are holomorphic in $\mathbb{C}^n_{++}$.
 \item (C2) : There is a nonempty region (an open set) $B\subset
 \mathbb{C}^n_{++}$ such that \eqref{eq:lpindep} holds on $B$.  
\end{itemize}
Then, due to the identity theorem e.g. \cite[Theorem
 4.1, Ch.1]{grauert:fritzsche:1976}, we conclude that $g_1=g_2$ for all 
 $\bfw\in \mathbb{C}^n_{++}$. Thus by the continuity of $g_i$  on exes
 and at the origin,
 \eqref{eq:lpindep} holds for all $\bfw\in\R^n_+$. 

First we check (C1). Since 
\[
 \E \sum_{j=1}^n \eta(t_{j-1},\eta_j] \prod_{k=1}^n e^{-w_k
 \eta(t_{k-1},t_k]} 
\]
exists for any $\bfw \in \bbC^n_{++}$, we apply  
the dominated convergence for changing order of derivative and expectation,
 and obtain 
\begin{align*}
 \frac{\partial g_1}{\partial \overline w_j} = \E\frac{ \partial
  \prod_{k=1}^n e^{-w_k \eta(t_{k-1},t_k]}}{\partial \overline w_j}=0,
\end{align*}
where the partial derivative is done with respect to complex conjugate
 $\overline{w}_j$ of $w_j$. Similarly we obtain $\partial g_2/ \partial
 \overline{w}_j=0$ on $\bbC^n_{++}$. Thus holomorphicity follows
 from e.g. \cite[Theorems 3.8 and 6.2, Ch.1]{grauert:fritzsche:1976}. 

 Condition (C2) is technical since \eqref{eq:lpindep} holds only on
 a subset in $\R^n_+$ but we need to extend the domain to
 $\bbC^n_{++}$. Notice that in 
 \eqref{eq:LPindep} we could replace $u_k$ with $u_k-iv_k,\,v_k\in \R$,
 so that \eqref{eq:LPindep} holds with 
 \[
  w_k = \int_{\R_+}(e^{-(u_k-iv_k)x}-1)\nu(dx)=: f(u_k,v_k),\,k=1,\ldots,n.
 \]
 We study the function $f$ and further write 
 \[
  f(u,v) = \int_{\R_+} (e^{-ux}\cos vx-1)\nu(dx) +i \int_{\R_+} e^{-ux}\sin vx
 \nu(dx)=: f_1(u,v)+if_2(u,v),\quad u\ge 0,\,v\in \R
 \]
 and apply the inverse function theorem to $f$ in order to show that 
 the range of $f$ could be open in $\bbC_{++}$. The Jacobian of $f$ is
 calculated as 
 \begin{align*}
  J_f(u,v)= \left|
 \begin{array}{cc}
  \frac{\partial f_1}{\partial u} & \frac{\partial f_1}{\partial v} \vspace{2mm}\\
  \frac{\partial f_2}{\partial u} & \frac{\partial f_2}{\partial v}
 \end{array}
\right| = -\big( 
\int_{\R_+} e^{-ux}x\cos (vx)\nu(dx)
 \big)^2 -\big( 
\int_{\R_+} e^{-ux}x\sin (vx)\nu(dx)
 \big)^2, 
 \end{align*}
where derivatives under the integral with $\nu$ are assured by
 \eqref{lh:represent}. Since the
 right-hand side approaches $-|\int e^{-ux}x \nu(dx)|^2$ as $v\to0$,
 which is not zero, 
 there exists a point $(u_0,v_0)\in \bbC_{++}$ such that $J_f(u_0,v_0)\neq
 0$. Now by the inverse mapping theorem, $f$ maps a neighborhood of
 $(u_0,v_0)$ to some neighborhood of $f(u_0,v_0)$ bijectively, so that
 we may take the neighborhood as an open set in $\bbC_{++}$. 
 Since the $n$th product of open sets in $\bbC_{++}$ constitute an open
 set in $\bbC_{++}^n$, (C2) is satisfied. 

The proof of $(ii)$ is quite similar to the case $(i)$ and we omit it. 
\end{proof}

\begin{remark}
 Theorem \ref{thm:main1} is a special case of Theorem
 \ref{thm:general} where $L$ is
 Poisson. Theorem \ref{thm:general} covers other famous point processes e.g. compound Cox
 processes. An interesting future topic is to extend the subordinator
 $L$ to more general L\'evy
 processes. 
\end{remark}

Next we investigate distribution of $((T_j)_{j:T_j\le t},N(t))$, which 
in view of \eqref{eq:condidist} seems to be intractable. 
Our strategy is to 
recall
the equivalence between a mixed Poisson process and mixed sample process
on a finite interval 
(see
e.g.\,\cite[Theorem 6.6]{grandell:1997} or \cite[Corollary 
9.2]{kallenberg:1983}, see also \cite{matthes:kerstan:mecke:1978}), 
namely given the number the points of a mixed Poisson process can be
regarded as iid non-ordered ones which are indeed those of the
corresponding mixed sample process. 
In our case under conditioning on $\eta$, we could regard arrival times of a
Cox process as conditionally iid r.v.'s given $N$, so that given $\eta$, the joint
distribution of the arrival times and number is available. Then we 
remove conditioning on $\eta$ by taking expectation. A similar technique is found in
the exact mixed Poisson case (see \cite[Lemma 1.2.1.1]{matsui:rolski:2016}). 

In what follows denote the iid non-ordered points and number of the process
by $((T_j')_{j:T_j'\le t},N(t))$ for fixed $t>0$ and we obtain an analytic
expression of the joint distribution. 
\begin{proposition}
\label{prop:jointdist}
 Let $N$ be a Cox process directed by non-decreasing additive process
 $\eta$. Then for non-ordered points $(T'_j)$ of $N(t),\,t>0$, the joint distribution of $((T_j')_{j:T_j'\le t},N(t))$ is 
\begin{align}
\label{ex:iidjoint}
 \P(T_1' \le t_1,\ldots,T_n' \le t_n,\,N(t)=n) =
 \frac{1}{n!}\E \prod_{j=1}^n \eta(t_j) e^{-\eta(t)},\quad t_j\in
 (0,t]. 
\end{align}
 Since $(T_j')$ are iid non-ordered, without loss of
 generality we let $0<t_1\le t_2\le \cdots \le t_n \le t$. Then
 \eqref{ex:iidjoint} has an expression 
\begin{align*}
 \sum_{\substack{0\le k_j \le n+1-j-\sum_{i>j}^n k_i, \\ k_1+\cdots +k_n=n}} \prod_{j=1}^n
 \frac{(n+1-j-\sum_{i>j}^n k_i)!}{k_j!(n+1-j-\sum_{i=j}^n k_i)! }\, 
 \E \eta(t_{j-1},t_j]^{k_j} e^{-\eta(t_{j-1},t_j]}\cdot \E
 e^{-\eta(t_n,t]},  
\end{align*}
 when $t_i\neq t_j$ for $i\neq j$ and when $t_{j-1}=t_j$ for some $j$ we
 let $k_j=0$ and $0^0=1$ in the above. 
\end{proposition}

\begin{proof}
 The expression \eqref{ex:iidjoint} is immediate from the conditional
 argument. For the second expression we write \eqref{ex:iidjoint} as 
\begin{align*}
 \eta(0,t_1](\eta(0,t_1]+\eta(t_1,t_2])\cdots(\eta(0,t_1]+\cdots+\eta(t_{n-1},t_n])
 e^{-\eta(t)}, 
\end{align*}
 and expand this to write by the polynomials of
 $\eta(t_{j-1},t_j]^{k_j}e^{-\eta(t_{j-1},t_j]}$ and $e^{-\eta(t_n,t]}$ with
 $j=1,\ldots,n$ where $k_j\le n+1-j,\,k_j\in \Z_+$ are powers of
 $\eta(t_{j-1},t_j]$. When $t_{j-1}=t_j$ the increment
 $\eta(t_{j-1},t_j]$ disappears and hence $k_j$ should be $0$.  
\end{proof}

To calculate the expression in Proposition \ref{prop:jointdist} we
observe that 
\[
 \E \eta(t_{j-1},t_j]^\ell e^{-u \eta(t_{j-1},t_j]}= (-1)^\ell
 \phi_j^{(\ell)}(u), 
\]
where $\phi_j(u)=\E e^{-u\eta(t_{j-1},t_j]}$ and
$\phi_j^{(\ell)}(u),\,\ell=1,2,\ldots$ are derivatives of order $\ell$ at $u$. 
Usually derivatives of $\phi_j(u)$ are complicated. However, we can use
the following recursive formulas.

\begin{lemma}
 Let 
\[
 \phi_j(u)=\E e^{-u \eta(t_{j-1},t_j]}=e^{\int_{(t_{j-1},t_j]\times
 \R_+}(e^{-uy}-1)\rho (d (x,y))},
\]
then we have 
\[
 \phi_j^{(\ell)}(u) = \sum_{i=1}^\ell (-1)^{i-1} \binom{\ell-1}{i-1} \int_{(t_{j-1},t_j]\times \R_+} y^i e^{-uy}
 \rho (d (x,y))\cdot \phi_j^{(\ell-i)}(u).
\]
\end{lemma}

\begin{proof}
 We just apply the Leibniz rule to 
 \[
  \phi_j'(u)=(-1) \int_{(t_{j-1},t_j]\times \R_+} y e^{-uy} \rho (d
 (x,y))\cdot \phi_j(u), 
 \]
where $\rho$ is the measure given in \eqref{chf:additive:process}. 
\end{proof}

\begin{proposition}
 Let $N$ be a Cox process directed by a Gamma process $\eta$ with parameters
 $\gamma,\,\lambda>0$ such that $\eta(t)\sim \Gamma (\gamma
 t,\lambda)$ for fixed $t>0$. Without loss of generality we assume
 $0<t_1\le t_2 \le \cdots \le t_n \le t$. Then 
\begin{align}
 \P(T_1'\le t_1,T_2' \le t_2,\ldots,T_n' \le t_n, N(t)=n) &=
 \prod_{k=1}^n \left(
\frac{\gamma t_k + k-1}{\gamma t_{k+1}+k-1}\right) \cdot \frac{1}{n!} \E
 \eta(t)^n e^{-\eta(t)} \nonumber \\
 &= \prod_{k=1}^n \left(
\frac{\gamma t_k + k-1}{\gamma t_{k+1}+k-1}\right) \cdot \frac{1}{n!}
 \frac{\Gamma(\gamma t+n)}{\Gamma (\gamma t)}
 \frac{\lambda^\alpha}{(1+\lambda)^\alpha}. \label{eq:gammaproc}
\end{align}
 From \eqref{eq:gammaproc}, it is immediate to see 
\begin{align*}
 \P(T_1'\le t_1,\ldots,T_n' \le t_n \mid N(t)=n) = \prod_{k=1}^n \left(
\frac{\gamma t_k + k-1}{\gamma t_{k+1}+k-1}
\right).
\end{align*} 
\end{proposition}

\begin{proof}
 By conditioning on $\eta$, we have 
 \begin{align}
  & \P(T_1' \le t_1,\ldots,T_n' \le t_n, N(t)=n \mid \eta(u), 0<u \le
  t ) \nonumber \\
  &= \eta(0,t_1]\cdot \eta(0,t_2] \cdots \eta(0,t_n]\cdot \frac{1}{n!}
  e^{-\eta(t)} \nonumber \\
  &= \frac{\eta(0,t_1]}{\eta(0,t_2]}\cdot \left(
 \frac{\eta(0,t_2]}{\eta(0,t_3]}
\right)^2 \cdot \left(
 \frac{\eta(0,t_3]}{\eta(0,t_4]}
\right)^3 \cdots \left(
\frac{\eta(0,t_n]}{\eta(t)}
\right)^n \cdot \frac{\eta(t)^n}{n!} e^{-\eta(t)}. \label{eq:condieta1}
 \end{align}
For our purpose it suffices to show that 
\begin{align*}
 \frac{\eta(0,t_1]}{\eta(0,t_2]},\, \frac{\eta(0,t_2]}{\eta(0,t_3]},\,
 \cdots, \frac{\eta(0,t_n]}{\eta(t)},\,
 \eta(t)
\end{align*}
are totally independent, and 
\begin{align*}
 \frac{\eta(0,t_{k}]}{\eta(0,t_{k+1}]}\sim \mathrm{Beta}
 (\gamma t_k,\gamma(t_{k+1}-t_k)),\quad 1\le k\le n,\quad t_{n+1}:=t. 
\end{align*}
The latter result follows from the property of Gamma r.v.'s. 
For the former result, 
we use the induction and assume the relation:  
\begin{align}
\label{induction:indep}
 ``\frac{\eta(0,t_1]}{\eta(0,t_2]}, \cdots,
 \frac{\eta(0,t_{\ell}]}{\eta(0,t_{\ell+1}]},\, \eta(0,t_{\ell+1}] \quad \text{are totally independent''}
\end{align}  
holds for $\ell=k-1$ and show that it holds also for $\ell=k$. 
Due to the property of Gamma r.v.'s 
\begin{align*}
 \frac{\eta(0,t_{k}]}{\eta(0,t_{k+1}]}\quad \text{and}\quad \eta(0,t_{k+1}]\quad \text{are
 independent for all}\ \ell \le n,  
\end{align*}
and the induction hypothesis with $\ell=1$ holds obviously. 
Since $\eta$ is a L\'evy process $\eta(s,t]$ is independent of the
 filtration $\mathcal F_s$ for all $0\le s<t<\infty$. 
Then 
\begin{align}
\label{indep:quatient}
 \frac{\eta(0,t_1]}{\eta(0,t_2]}, \cdots,
 \frac{\eta(0,t_{k-1}]}{\eta(0,t_k]},\, \eta(0,t_k],\,\eta(t_{k},t_{k+1}]\quad \text{are totally independent}.
\end{align}
Since a random set $(\eta(0,t_k]/\eta(0,t_{k+1}],\,\eta(0,t_{k+1}])$ in
 \eqref{induction:indep} is
 included in $\sigma$-field by $\{\eta(0,t_k],\,\eta(t_k,t_{k+1}]\}$, it 
is independent of $\eta(0,t_{k-1}]/\eta(0,t_k]$ by
 \eqref{indep:quatient} together with e.g. 
 \cite[Theorem 3.3.2]{chung:2000}. 
Now keeping \eqref{indep:quatient} in mind, we apply the
 relation between pairwise and total independence (
 \cite[Lemma 3.8]{kallenberg:1983}) from the right-hand side of
 \eqref{induction:indep} with $\ell=k$. This yields the desired total independence for $\ell=k$.
\end{proof}

\section{Prediction in Cox cluster processes}\label{sec:prediction}
 As an application 
 we consider a prediction
 problem of the model \eqref{eq:cpp} given the past information, 
 assuming that $\eta$ is a non-decreasing additive process and $L$ an
 additive process. 
 Such prediction problems are studied lately e.g. in 
\cite{LBPJ:1999,matsui:mikosch:2010,rolski:tomanek:2011,rolski:tomanek:2014,matsui:rolski:2016,matsui:2016} (see also references therein) motivated by 
 a non-life insurance application. 
 In the model \eqref{eq:cpp}, $T_j$ may describe the arrival of a
 claim in an insurance portfolio and $L_j(t-T_j)_{t\ge T_j}$ is the corresponding payment process from
 the insurer to the insured starting at time $T_j$. 
 This interpretation of the process has been propagated by Norberg
 \cite{norberg:1993} (cf. \cite{norberg:1999}). However, note that the shot noise process
 \eqref{eq:cpp} has a variety of applications: finance, hydrology,
 computer networks, queuing theory, etc. and our method here is also 
 applicable in other contexts.

 For notational convenience with regards $L$, 
 we define kinds of mean value functions, 
 \begin{align*}
  \mu(s,t]= \E L(s,t]\quad \text{and}\quad \sigma^2(s,t]:= \E L^2(s,t]-
  \mu^2(s,t],\quad t>s\ge 0. 
 \end{align*}
 We also write $\mu(t):=\E L(0,t]$ and $\sigma^2(t):=\E
 L^2(0,t]-\mu^2(t)$. Throughout we assume that stochastic integrals with
 $\eta$,
 \[
  \int_{(0,t]}\mu(t-u)\vee \mu^2(t-u)\eta(du)\quad \mathrm{and}\quad
  \int_{(0,t]}\sigma^2(t-u)\eta(du), 
 \]
 exist in the sense of definition in \cite[p.11]{rajput:rosinski:1989}. 
 Here we do not pursue the detailed integrability condition by $\eta$,
 which you could find in \cite[Theorem 2.7]{rajput:rosinski:1989}, since
 our main purpose is an application of the previous results. 

  Basic property the model \eqref{eq:cpp} is as follows. These moments
 are calculated by using the characteristic function of the stochastic
 integral with $\eta$ (cf. \cite[Proposition 2.6]{rajput:rosinski:1989}). 
\begin{proposition}
 Assume the model \eqref{eq:cpp} with $\eta$ a non-decreasing additive process. 
 Then for $s,t>0$
 \begin{align*}
  \E M(t) & = \E \int_{(0,t]}\mu(t-u)\eta(du) =  \int_{(0,t]\times \R_+} \mu(t-u) x \rho(d(u,x)), \\
  \mathrm{Cov}(M(s)M(s+t)) &= \E
  \int_{(0,s]}(\mu^2+\sigma^2)(s-u)\eta(du) + \E \int_{(0,s]\times(0,s]}
  \mu(s-u)\mu(s+t-v) \eta(du)\eta(dv) \\
  &\quad - \E M(s) \E M(s+t) \\
 &= \int_{(0,s]}\sigma^2(s-u) x \rho(d(u,x))+  \int_{(0,s]\times \R_+}\mu(s-u) \mu(s+t-u)(x^2+x) \rho(d(u,x)).  
 \end{align*}
\end{proposition}
 Notice that from the covariance function, we know that $M$ does not have
 independent increments. The next result gives expressions of the
 predictor and its conditional mean squared error.
\begin{theorem}
 \label{thm:prediction}
 Let $N$ be a Cox process directed by a non-decreasing additive process
 $\eta$, and $\mathcal G_s$ denote the $\sigma$-field by
 $\{N(s),\,(T_j)_{j:T_j\le s},\,(L_j(t-T_j))_{j:T_j\le s}\}$. Then the
 process $M$ by \eqref{eq:cpp} satisfies 
 \begin{align}
 \E[M(s,s+t]\mid \mathcal G_s] &= \sum_{j=1}^{N(s)}\mu(s-T_j,s+t-T_j] +
  \E \int_{(s,s+t]} \mu(s+t-x) \eta (dx), \label{condexp:predi}\\
 &= \sum_{j=1}^{N(s)}\mu(s-T_j,s+t-T_j] + \int_{(s,s+t]\times \R_+}
  \mu(s+t-u)x \rho(d(u,x)), \nonumber \\
 \mathrm{Var}(M(s,s+t]\mid \mathcal G_s) &=
  \sum_{j=1}^{N(s)}\mu(s-T_j,s+t-T_j] + \mathrm{Var} \big(
 \int_{(s,s+t]}\mu(s+t-x)\eta(dx) 
 \big)\label{condvar:predi} \\
 & \quad + \E \int_{(s,s+t]}(\mu^2 +\sigma^2)(s+t-x)\eta(dx)
  \nonumber \\
 & = \sum_{j=1}^{N(s)}\mu(s-T_j,s+t-T_j] + \int_{(s,s+t]\times\R_+} 
 \mu^2(s+t-u)x^2 \rho(d(u,x)) \nonumber \\
 &+\int_{(s,s+t]\times\R_+} (\mu^2+\sigma^2)(s+t-u)\rho(d(u,x)), \nonumber 
 \end{align}
 where $\mu$ and $\sigma^2$ are respectively mean and variance functions
 of $L$ and $\mu^2(s+t-x):= (\mu(s+t-x))^2$. 
\end{theorem}

\begin{proof}
Let $\mathcal H_{s+t}$ be the $\sigma$-field by $(T_j)_{j:T_j\le
 s+t},N(s),N(s+t)$ and $(L_j(t-T_j))_{j:T_j\le s+t}$, so that $\mathcal G_s
 \subset \mathcal H_{s+t}$. Write 
 \[
  M(s,s+t] = \sum_{j=1}^{N(s)} L_j (s-T_j,s+t-T_j] +
 \sum_{j=N(s)+1}^{N(s+t)}L_j (s+t-T_j)
 \]
 and take conditional expectation on $\mathcal G_s$,
 \begin{align*}
  & \E [\sum_{j=1}^{N(s)}L_j(s-T_j,s+t-T_j] \mid \mathcal G_s] + \E\,\E
  [\sum_{j=N(s)+1}^{N(s+t)} L_j(s+t-T_j)\mid \mathcal H_s] \mid \mathcal
  G_s ] \\
 &= \sum_{j=1}^{N(s)} \mu(s-T_j,s+t-T_j]+ \E [\sum_{j=N(s)+1}^{N(s+t)} \mu(s+t-T_j)], 
 \end{align*}
 where we use the repeated expectation \cite[Theorem 6.1 (vii)]{kallenberg:2002} argument together with Theorem
 \ref{thm:main1}. In the last expression, notice that the sequence $(T_j)_{N(s)+1
 \le j\le N(s+t)}$ is symmetric and given $N$ and $\eta$, we could
 regard it as an iid sequence such that $T_j\sim
 \eta(dx)/\eta(s,s+t],\,s<x\le s+t\, a.s.$ Hence a conditional
 argument gives the first part of \eqref{condexp:predi}. The second
 expression of \eqref{condexp:predi} is obtained by differentiating LP
 of the stochastic integral with $\eta$ (cf. \cite[Proposition 2.6]{rajput:rosinski:1989}). 

 For the expression \eqref{condvar:predi}, we write 
 \begin{align*}
  (M(s,s+t])^2 &= \big( \sum_{j=1}^{N(s)} L_j(s-T_j,s+t-T_j)
\big)^2 + 2 \sum_{j=1}^{N(s)}
  L_j(s-T_j,s+t-T_j]\sum_{j=N(s)+1}^{N(s+t)}L_j(s+t-T_j) \\
 &+ \big(
 \sum_{j=N(s)+1}^{N(s+t)} L_j(s+t-T_j)
 \big)^2 =: \mathrm{I} + 2\mathrm{II} +\mathrm{III} 
 \end{align*}
 and take conditional expectations for these quantities, where the
 repeated expectation argument together with Theorem \ref{thm:main1} are
 again used. Then we obtain 
 \begin{align*}
  \E [\mathrm{I} \mid \mathcal G_s] &= \sum_{j=1}^{N(s)} (\mu^2+\sigma^2)
 (s-T_j,s+t-T_j] + \sum_{j\neq k,\,j,k=N(s)+1}^{N(s)} \mu(s-T_j,s+t-T_j]
  \mu(s-T_k,s+t-T_k], \\
  \E [\mathrm{II} \mid \mathcal G_s] &= \sum_{j=1}^{N(s)}
  \mu(s-T_j,s+t-T_j] \E [\sum_{j= N(s)+1}^{N(s+t)} \mu(s+t-T_j)], \\
 \E [\mathrm{III} \mid \mathcal G_s] &= \E [
 \sum_{j=N(s)+1}^{N(s+t)} (\mu^2+\sigma^2)(s+t-T_j)] + \E [
 \sum_{j\neq k,\,j,k=N(s)+1}^{N(s+t)} \mu(s+t-T_j) \mu(s+t-T_k) 
],
 \end{align*}
 so that 
 \begin{align*}
  \mathrm{Var} (M(s,s+t] \mid \mathcal G_s) &= \E
  [\mathrm{I}+2\mathrm{II}+\mathrm{III} \mid \mathcal G_s]-\big(
 E [M(s,s+t] \mid \mathcal G_s]
\big)^2 \\
 &= \sum_{j=1}^{N(s)}\sigma^2 (s-T_j,s+t-T_j)+ \E [
 \sum_{j=N(s)+1}^{N(s+t)} (\mu^2+\sigma^2)(s+t-T_j)]\\
 &\quad + \E [\sum_{j\neq k,\,j,k=N(s)+1}^{N(s+t)}
 \mu(s+t-T_j)\mu(s+t-T_k)
 ]  -\big( \E [\sum_{j=N(s)+1}^{N(s+t)}\mu(s+t-T_j) ]
 \big)^2.
 \end{align*}
 Now since the conditional order statistic property of $(T_j)$ yields 
 \begin{align*}
  \E [
 \sum_{j=N(s)+1}^{N(s+t)}(\mu^2+\sigma^2)(s+t-T_j)
 ] &= \E [\int_{(s,s+t]}(\mu^2+\sigma^2)(s+t-u)\eta(du)], \\
  \E [\sum_{j\neq k,\,j,k=N(s)+1}^{N(s+t)} \mu(s+t-T_j)] &= 
 \E [ \big( \int_{(s,s+t]}\mu(s+t-u)\eta(du) \big)^2], 
 \end{align*}
we obtain \eqref{condvar:predi}. The second expression is obtained again
 by $LP$ of the stochastic integrals with $\eta$. 
\end{proof}

By taking expectation of \eqref{condvar:predi}, we could evaluate the squared
error of the prediction. 
\begin{corollary}
Under the assumption of Theorem \ref{thm:prediction}, the unconditional
 squared error
 of the prediction is 
 \begin{align*}
   E\big( &M(s,s+t]-E[M(s,s+t] \mid \mathcal G_s] \big)^2 \\ 
  &\quad = \E \int_{(0,s]} \sigma^2(s-u,s+t-u] \eta(du) + \E \int_{(s,s+t]}
  (\mu^2+\sigma^2)(s+t-u)\eta(du) \\
  &\qquad + \mathrm{Var}\big( \int_{(s,s+t]}\mu(s+t-u)\eta(du)
 \big) \\
  &\quad = \int_{(0,s+t]\times \R_+} \sigma^2(s-u,s+t-u] x
  \rho(d(u,x)) \\
  &\qquad + \int_{(s,s+t]\times \R_+} \{ (\mu^2+\sigma^2)(s+t-u)x +
  \mu^2(s+t-u)x^2 \} \rho(d(u,x)). 
 \end{align*}
\end{corollary}
The proof is obvious from that of Theorem \ref{thm:prediction} and we
omit it. 

\section{Numerical example}
We consider a numerical example for 
the model \eqref{eq:cpp} and examine the prediction procedure in the
previous section.
For the underling random measure $\eta$ of $N$, we suppose 
a homogeneous Poisson process with parameter $10$ so that it is a
subordinator, and for a generic process $L$ of $(L_j)$ we assume the non-homogeneous
Poisson with directing measure $\mu(0,t]:=5(1-e^{-t})$ which is proportional to
d.f. of the exponential r.v.  
In Figure \ref{fig:1} (left), we illustrate the process $M$ by \eqref{eq:cpp}
for the interval $t\in[0,5]$. Dots of $\square$ are arrival
times of Cox process $N$ where multiple jumps are allowed since the mean
measure $\eta$ is from Poisson so that it has atoms. Plots by $\bullet$ are
points by processes $(L_j)$ triggered by arrival times $(T_j)$. 
The set of points $\bullet$ from each $L_j$ is written in the same
horizontal axis as that of $T_j$. 
For points $\bullet$ at
$y=0$ imply that no arrival points from corresponding
$L_j$ are observed in $[0,5]$. 

From the results of the previous section, predictor is  
explicitly obtained,
\[
 \E [ M(s,t] \mid \mathcal G_s ] = 5 \sum_{j=1}^{N(s)}
 e^{-(s-T_j)}(1-e^{-(t-s-T_j)}) + 10(t-s-1+e^{-(t-s)}). 
\] 
In Figure \ref{fig:1} (right), the total number of $M(t),\,t\in(0,5]$ is plotted by
dots $\bullet$. Other dots are predictions of $M$ on $(s,5]$ given the
information $\mathcal G_s$ before $s$. One could see that the more
previous information $\mathcal G_s$ we use, the better predictions we
have. This is possible since $M$ does not succeed to independent increments
of $N$ or $L$ 
any more. In view of Figure \ref{fig:1}, our procedure seems to work
reasonably. 

Notice that Norberg in \cite{norberg:1999} studied $\E[M(t, t + s] | \mathcal F_t]$ with
$\mathcal F_t$ to be the full history when $N$ is a simple Poisson, and
obtained their explicit expressions. However, when $N$ is a Cox process 
Norberg 
suggested the inhomogeneous 
linear prediction method. Here we show that even when $N$ is a Cox
process we could obtain explicit expressions for predictors with the
full past information. For other prediction method 
with conditional expectations we refer to \cite{LBPJ:1999}, of which settings are different from ours.

\begin{figure}
\begin{center}
\includegraphics[width=0.45\textwidth]{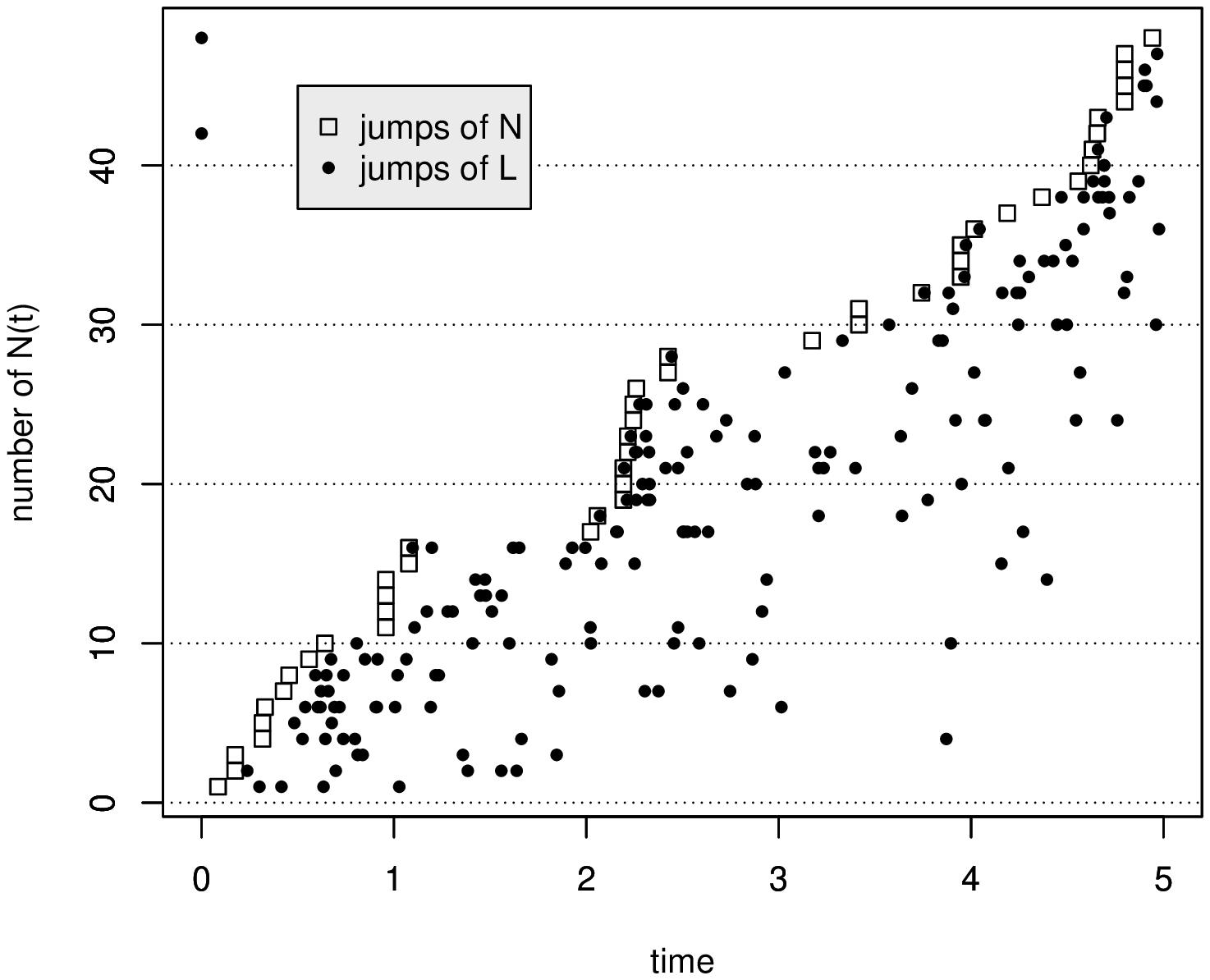}
\includegraphics[width=0.45\textwidth]{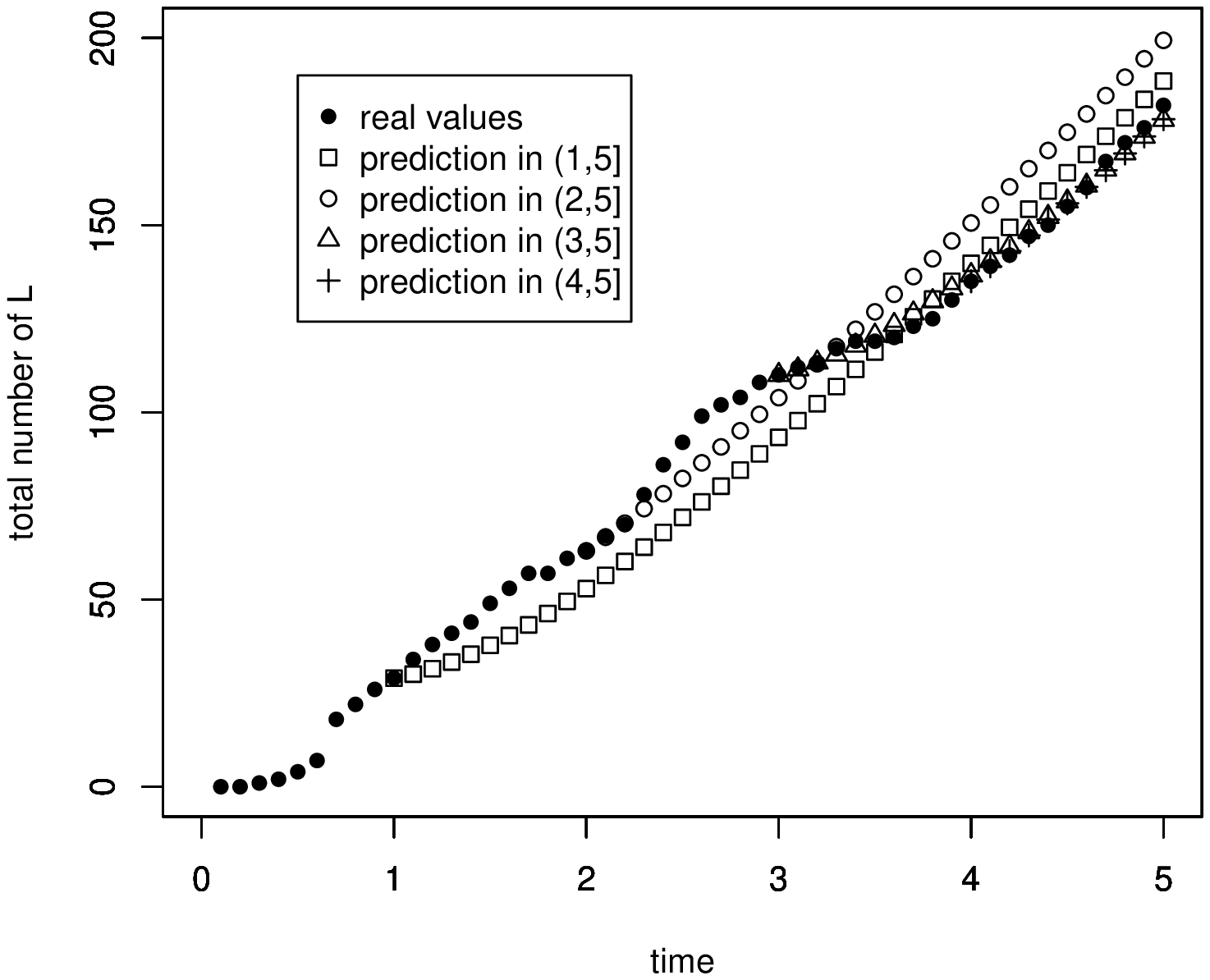}
\end{center}
\caption{Left: we present the model $M(t),\,t\in [0,5]$ of \eqref{eq:cpp} with  
$L$ to be a non-homogeneous Poisson, where processes $N$ and $(L_j)$ are plotted separately. Dots by
 $\square$ denote jumps $(T_j)$ of $N$ and dots by $\bullet$ on the same
 $y$-coordinate as that of $T_j$ are the corresponding jumps by $L(t-T_j)$. The dots
 $\bullet$ on the $y$-axis suggest that no jumps by the stream
 $L(t-T_j)$ are observed for these $j$'s on 
 $[0,5]$.  Right: we examine the predictors of $M(t),\,t\in [0,5]$ for
 different intervals. Dots $\bullet$ imply the real observation of
 $M(t),\,t\in [0,5]$, while sequences of dots $\square\sim +$
 respectively denote predictors $\E[M(s,5]\mid \mathcal
 G_s],\,s=1,2,\ldots,4$. 
}
\label{fig:1}
\end{figure}

\end{document}